\documentclass[12pt,a4j]{article}

\input{1203convexoid.STY}
\makeindex
\begin{document}
\title{Compactifying $\Spec \ZZ$}
\author{Satoshi Takagi}
\date{}
\maketitle

\begin{abstract}
In this paper, we introduce a new algebraic type
of `convexoid rings', and we give the definition
of (weak) convexoid schemes, which share
similar properties with ordinary schemes.
As a result, we give a purely-algebraic construction of
the compactification $\overline{\Spec \ZZ}=\Spec \ZZ \cup \{\infty\}$,
which is realized as the Zariski-Riemann space of $\Spec \ZZ$
in the category of weak convexoid schemes.
\end{abstract}

\tableofcontents

\setcounter{section}{-1}
\section{Introduction}
In this paper, we give a purely-algebraic 
construction of the compactification 
$\overline{\Spec \ZZ}=\Spec \ZZ \cup \{\infty\}$.

The philosophy of Arakelov tells 
that the correct compactification of $\Spec \ZZ$
should be the space which consists of finite places
together with the infinite place $\infty$.
However, the conventional theories could
not obtain this space canonically, since $\Spec \ZZ$
is the final object in the category of schemes.
Therefore, Arakelov geometers and
number theorists had to give ad hoc definition
for the desired spaces:
in Arakelov geometry, we endow an hermitian
metric on vector bundles, as a substitute
for the information on the infinite place;
in number theory, places (finite or infinite)
are defined by valuations, and is not defined
algebraically.

Here it might be valuable to ask
\textit{why} these concepts and definitions behave so nicely,
not \textit{how}.
Also, we might ask why the infinite place
cannot be realized within the category of schemes.

The key point is simple.
Here, we give a new type of algebra
which we call \textit{convexoid rings}:
these have two binary operators $\boxplus$ and $\times$,
and are (commutative) monoids with respect to $\times$:
however, we do not assume the associativity of $\boxplus$.
The category of convexoid rings contains
that of rings as a full subcategory
(and also, multiplicative monoids with absorbing elements),
and we can consider `convexoid schemes'
as a generalization of schemes.
We can go further, and define `weak convexoid schemes'
so that we can treat Zariski-Riemann spaces properly.
As a corollary, we obtain the main theorem:
\begin{Thm}
The compactification $\overline{\Spec \ZZ}=\Spec \ZZ \cup \{\infty\}$
of $\Spec \ZZ$ can be realized as a weak convexoid scheme.
It is defined by the universal property, namely
the Zariski-Riemann space of $\Spec \ZZ$ over $\Proj R_{0}$,
where $R_{0}$ is the initial object in the category
of convexoid rings.
The stalk $\scr{O}_{\infty}$ of $\overline{\Spec \ZZ}$
at the infinity place is the valuation convexoid ring $\mathbf{D}\QQ$,
which is the unit disk in $\QQ$ consisting of rationals
the absolute value of which is not more than $1$.
\end{Thm}
This theorem can also be extended to the
ring of integers $\mathcal{O}_{K}$ of any algebraic field $K$.

We remark that the set $\Gamma(\overline{\Spec \ZZ},\scr{O})$
of global sections is $\{0,\pm 1\}$, which some
$\FF_{1}$-geometers denote by $\FF_{1^{2}}$.
This does not have the $\boxplus$-structure,
but only the multiplicative monoid structure as expected.
We dare not say that we have obtained the
correct definition of $\FF_{1}$ (or $\FF_{1^{2}}$);
many people are hoping for too many dreams on $\FF_{1}$,
and we just gave a partial answer for this.

We also remark that this compactification of $\Spec \ZZ$
is almost identical to that of Haran's \cite{Haran}, 
or even that of Durov's \cite{Durov};
of course they are not mentioning the convexoid structure,
but at least the stalks of the structure sheaves on the infinite places
coincide as a multiplicative monoid.
However, we emphasize the fact
that the construction given in this paper
canonically induces the archimedean norm
structure from the algebraic structure,
and therefore $\overline{\Spec \ZZ}$
is determined by the universal property;
while the other two bring the archimedean
norm structure outside the algebraic world
and therefore their definitions are \textit{ad hoc}.

This paper is organized as follows:
In \S 1, we illustrate how we come up with convexoids,
since the reader may wonder why he or she
has to be involved with it instead of sticking to
the classical world of rings.

In \S 2, we give the definition of multi-convexoids
and multi-convexoid rings,
and see that their behaviour is quite similar to those of rings.

In \S 3, we prove a variant of the classical
Ostrowski's theorem.
The crucial difference is that it is formulated 
in completely algebraic terms, and
this theorem lies at the heart of the main result of this paper.

In \S 4, we give the
definition of convexoid schemes.
It is already assured that we can define them
analogously as the theory of ordinary schemes
\cite{Takagi1}. However to reach the goal,
we must weaken the condition of what a
`patching' should be, by admitting certain kinds
of twists, or in other words,
weak homomorphisms.
A twist does not affect the semiring of ideals,
hence we can safely run the construction
of the underlying space of the spectrum.
We also give the construction of the
`fake closure' of $\Spec \ZZ$:
this is a convexoid scheme with the underlying space
homeomorphic to $\overline{\Spec \ZZ}$,
and is close to our answer.
Still, it is mal-behaved on the infinite place,
hence we will seek for further improvement
in the latter sections.

In \S 5, we give the definition of graded convexoid rings,
and the convexoid schemes $\Proj A$ for a graded
convexoid ring $A$.
These constructions are completely analogous 
to those of rings, and we claim that the
above `fake closure' can be expressed as
$\Proj R_{0}$, which turns out to be a very natural object.

In \S 6, we define the notion
of weak convexoid schemes, which
is a variant of weak $\scr{C}$-schemes 
introduced in \cite{Takagi1}.
This enables us to treat Zariski-Riemann spaces,
and as a result, we finally reach the correct
definition of $\overline{\Spec \ZZ}$.

The last section \S 7 is an appendix,
which shows that the analogy of
linear systems and projective morphisms
in algebraic geometry is also valid
for $\Spec R_{0}$, and
that we have an immersion $\Proj R_{0} \to \PP$,
where $\PP$ is a proprojective space over $\FF_{1^{2}}$.
The concepts and
statements introduced in this section are by no means precise:
these will be affirmed in the forthcoming papers.

\textbf{Notation and conventions:}
Any ring is unital.
We denote by $\cat{CMnd$_{0}$}$ (resp. $\cat{CRing}$)
the category of commutative monoids with absorbing elements
(resp. commutative rings)
and their homomorphisms.
For any subring $R$ of $\CC$,
we denote by $\mathbf{D}R$ the unit disk
$\{x \in R \mid |x| \leq 1\}$.
This has a structure of a convexoid ring
(see Definition \ref{def:convexoid} and Theorem
\ref{thm:ostrowski}).
When given a commutative (convexoid) ring $R$,
we denote by $\Omega(R)$
the distributive lattice of finitely generated ideals
of $R$ modulo the congruence $\mathfrak{a}^{2}=\mathfrak{a}$.
Two ideals $\mathfrak{a}$ and $\mathfrak{b}$
is equal in $\Omega(R)$ if and only if
$\sqrt{\mathfrak{a}}=\sqrt{\mathfrak{b}}$.

We frequently use the terminologies
of category theory, based on the textbook
\cite{CWM}.
The theory of convexoid schemes 
shares most of the part with the one already
exposited in \cite{Takagi1} and \cite{Takagi2};
we will not repeat the argument here,
and many basic facts will be referred to the above mentioned
articles.

\section{Preliminary Observations}

The development of the theory
of schemes over $\FF_{1}$ arose recently,
motivated by the Riemann hypothesis.

Weil's conjecture,
which is an analogy of the Riemann hypothesis
for the positive characteristic case,
has been proved by Deligne \cite{Deligne}, 
by considering the multiple zeta function
on the $n$-fold product 
\[
X \times_{\FF_{q}} X \times_{\FF_{q}} \cdots \times_{\FF_{q}}
X
\]
of the given projective variety $X$ over $\FF_{q}$.

Many people are hoping to imitate this method
to prove the original Riemann hypothesis up to now.
The essential part is to find a correct `base field'
$\FF_{1}$, which is called \textit{the field with one element},
so that we can regard $\Spec \ZZ$
as an open curve defined over $\FF_{1}$,
and consider the $n$-fold object 
\[
\ZZ^{\otimes n}=\ZZ \otimes_{\FF_{1}} \ZZ \otimes_{\FF_{1}} \cdots
\otimes_{\FF_{1}} \ZZ
\]
with a multiple zeta function defined over it.
Although $\ZZ^{\otimes n}$ is not defined appropriately yet,
the preferred multiple zeta function is constructed
\cite{Kurokawa}.

Also, we would like to obtain the
compactification of $\Spec \ZZ$ over $\FF_{1}$
so that we could formulate Lefschetz-type
formula for the complete zeta function
\cite{Deninger}:
\[
\hat{\zeta}(s)=\prod_{i=0}^{2}\textstyle{\det_{\infty}}
\left(\Frac{1}{2\pi}
(s-\Theta)|H^{i}(X,\mathcal{R})\right)^{(-1)^{i+1}}.
\]
Connes has shown the determinantal representation
of the Riemann zeta function \cite{Connes1},
and furthermore gave a geometric representation,
by considering a function space on 
a projective line over $\FF_{1}$ \cite{Connes2}.

These results are suggesting the importance
of the theory of schemes over $\FF_{1}$,
apart from the philosophy of Arakelov.
However, the definition of $\FF_{1}$
has not reached a full agreement yet.
See \cite{Lorscheid} for the survey in this topic.

Let us go back to try for the compactification 
of $\Spec \ZZ$.

Recall that, when we are given a (non-compact) smooth
curve $X$ over a base field $k$,
we can construct its universal compactification
as a Zariski-Riemann space $\ZR(X,k)$ , namely
the spaces of valuation rings over $k$:
\[
\xymatrix{
\Spec k(X) \ar[r] \ar[d] & \overline{X}=\ZR(X,k) \ar[ld] \\
\Spec k
}
\]
Therefore, if we wish some analogy to hold,
$\overline{\Spec \ZZ}$ should then be the
Zariski-Riemann space of $\Spec \ZZ$ over $\FF_{1}$:
\[
\xymatrix{
\Spec \ZZ \ar[d] \ar[r] & \overline{\Spec \ZZ}=\ZR(\Spec \ZZ,\FF_{1})
\ar[dl] \\
\Spec \FF_{1}.
}
\]
However, since $\ZZ$ is the initial object in the category of rings,
we cannot have a `base field' $\FF_{1}$
in the category of rings;
we must widen our perspectives.

Some experts say that $\FF_{1}$-algebras should be
regarded as a monoid, and schemes over $\FF_{1}$
is a geometric object constructed from monoids.
One way to look at is
that $\FF_{1}$ is the initial element
in the category of commutative monoids with absorbing
elements: $\FF_{1}=\{0,1\}$.
(We can further attach an idempotent additive structure
so that $\FF_{1}$ becomes a Boolean algebra,
but this is not essential.)
However, it is doubtful that we can
recover the infinite place, only by considering
the multiplicative monoid structure:
indeed, we can define and consider Zariski-Riemann spaces
for a morphism of commutative monoids.
When applying this to $\FF_{1} \to \ZZ$,
we obtain a `proper space' $X$ over $\Spec \FF_{1}$.
However, this has infinitely many infinite places.
This happens since we ignore the additive structure,
and hence also the archimedean norm structure
of $\ZZ$.

This observation shows that we cannot totally abandon
the additive structure.

Let us look more closely.
On a finite place $p$ of $\ZZ$,
we obtain a local ring (more precisely,
a discrete valuation ring) $\ZZ_{(p)}$,
and by completion we obtain the ring $\ZZ_{p}$ of $p$-adic
integers.
This is the `unit disk' $\{x \in \QQ_{p} \mid |x|_{p} \leq 1\}$
of of the $p$-adic field $\QQ_{p}$.
If we wish to have an analogy of this
for the infinite place, then the objects
corresponding to $\ZZ_{(p)}$ (resp.
$\ZZ_{p}$, $\QQ_{p}$) should be
$\mathbf{D}\QQ=\{x \in \QQ \mid |x|_{\infty} \leq 1\}$,
(resp. $\mathbf{D}\RR=\{ x \in \RR \mid |x|_{\infty} \leq 1\}$,
$\RR$).
However, we run into a problem since
the unit disks $\mathbf{D}\QQ$ and $\mathbf{D}\RR$
are \textit{not} rings: they are multiplicative monoids,
but are not closed under addition.
 
This is the central motivation of introducing
a new algebra in this paper:
we want to have an algebraic type $V$ with a multiplicative
monoid structure such that,
\begin{enumerate}
\item
$V$-algebras share good properties with those of rings, and
\item
the category of $V$-algebras includes the unit disks
$\mathbf{D}\QQ$, $\mathbf{D}\RR$ shown above.
\end{enumerate}
The $V$-algebras are what we call \textit{convexoid rings}
in this paper.

Let us review $\mathbf{D}\QQ$.
This multiplicative monoid is not closed under
addition; however, we can always think of 
taking the mean value
$(a+b)/2$ of two elements $a,b \in \mathbf{D}\QQ$.
This binary operation $(a,b) \mapsto (a+b)/2$
will be denoted by $\boxplus$.
This operation does not satisfy associativity.
However, the distribution law holds,
and its behaviour resembles to that of rings very much.
Moreover, we can think of convexoid ring spectra
and schemes, just as in the case of rings.
This is because the general scheme theory
does not require associativity of the underlying operators \cite{Takagi1},
such as $\boxplus$.

Fortunately, the theorem of 
Ostrowski (Theorem \ref{thm:ostrowski})
tells that we can obtain the preferred valuations of $\QQ$
only by assuming the condition $1 \boxplus 1 \in R^{\times}$,
where $R$ is the valuation convexoid ring corresponding to
the valuation.
This condition corresponds to the triangular inequality:
$|a+b| \leq |a| +|b|$.
Hence, if we denote by $R_{0}$ the initial object
of convexoid rings, then the compactification $\overline{\Spec \ZZ}$
can be obtained over $R_{0}[(1\boxplus 1)^{-1}]$,
except for the finite place $p=2$, which is the antipode of
the infinite place.

Until now, we don't need multi-convexoid rings.
However, some problem arise when considering
projective convexoid schemes.
Let us review the construction $\Proj A$
for a commutative ring $A$.
$\Proj A$ is covered by the open sets of the form
$D_{+}(f)$, where $f$ is a homogeneous element of $A$,
and $D_{+}(f)$ is isomorphic to 
\[
\Spec A_{(f)}=\{ a/f^{n} \mid \deg a=n\deg f\}.
\]
When we try to apply this theory
to graded convexoid rings, we cannot give an
appropriate convexoid structure on $A_{(f)}$ when $\deg f >1$,
but only a multi-convexoid structure.
This is why we introduced the notion
of multi-convexoids.
However, multi-convexoids still behave
fairly well, and does not bother when constructing schemes.

This is the rough idea of the theory
introduced in this paper.

\section{Convexoids}

An algebraic type $V$ is \textit{commutative}
(in the sense of \cite{Takagi1}),
if for any $m$-ary operator $\phi$ and any $n$-ary operator
$\psi$, the following holds:
\begin{multline*}
\phi (\psi (x_{11},\cdots,x_{1n}),\cdots,\psi (x_{m1},\cdots, x_{mn})) \\
=\psi (\phi (x_{11},\cdots, x_{m1}),\cdots, \phi (x_{1n}, \cdots, x_{mn}) ).
\end{multline*}

\begin{Def}
\label{def:convexoid}
Let $d$ be a positive integer.
A \textit{$d$-convexoid} is a quadruple $(G,\boxplus^{d},-,0)$
where $G$ is a set, $\boxplus^{d}$ (resp. $-,0$) is a $2^{d}$-ary
(resp. unary, constant) operator on $G$ such that
\begin{enumerate}[(a)]
\item $G$ is commutative,
\item $\boxplus^{d}$ is symmetric,
namely 
\[
\boxplus^{d}(a_{\sigma(1)},\cdots,a_{\sigma(2^{d})})
=\boxplus^{d}(a_{1},\cdots,a_{2^{d}})
\]
for any element $\sigma$ of the symmetric group $\mathfrak{S}_{2^{d}}$.
\item $\boxplus^{d}(a_{1},\cdots,a_{2^{d-1}},-a_{1},\cdots,-a_{2^{d-1}})=0$
for any element $a_{1},\cdots,a_{2^{d-1}} \in G$.
\end{enumerate}
We denote by $\cat{Cxd$^{d}$}$ the category of $d$-convexoids
and its homomorphisms.
When $d=1$, we simply say `convexoids' and drop
the superscript.
Also, we write $a \boxplus b$ instead of 
$\boxplus^{1}(a,b)=\boxplus(a,b)$.
\end{Def}

By the commutativity,
the Hom set $\cat{Cxd$^{d}$}(M,N)$ canonically becomes
a convexoid for any convexoid $M,N$,
and the composition becomes bilinear.
Also, we can define the tensor product
$M \otimes(-):\cat{Cxd$^{d}$} \to \cat{Cxd$^{d}$}$ for any convexoid $M$
as the left adjoint of $\cat{Cxd$^{d}$}(M,-)$.
This gives a closed symmetric monoidal structure on $\cat{Cxd$^{d}$}$.
When we do not specify $d$,
we merely say multi-convexoids instead of $d$-convexoids.

A \textit{$d$-convexoid ring} is a monoid object
in $\cat{Cxd$^{d}$}$.
We denote by $\cat{CxdRing$^{d}$}$ the category
of $d$-convexoid rings and their homomorphisms.

\begin{Def}
Let $A$ be a $d$-convexoid ring.
\begin{enumerate}
\item The constant $\gamma_{A}=\boxplus^{d}(1,0,\cdots,0)$
is the \textit{fundamental constant} of $A$.
Note that this is in the center of $A$.
\item $A$ is \textit{normalized}, if $\gamma_{A}=1$.
\end{enumerate}
\end{Def}
This fundamental constant plays the key role in this algebra.

Let $(A,\boxplus^{d})$ be a $d$-convexoid ring,
and $u \in A$ be any element.
Then, we can define another $d$-convexoid structure
$\tilde{\boxplus}^{d}$ by setting
\[
\tilde{\boxplus}^{d}(a_{1},\cdots,a_{2^{d}})
=u\cdot \boxplus^{d}(a_{1},\cdots,a_{2^{d}}).
\]
This means that, we have many choices
of a $\boxplus^{d}$-structure on a convexoid ring $A$
and it is crucial to preserve this flexibility
when we consider convexoid schemes.

\begin{Def}
\label{def:cxd:weak:hom}
\begin{enumerate}
\item A map $f:A \to B$ between two $d$-convexoid
rings is a \textit{weak homomorphism} if 
\begin{enumerate}[(a)]
\item $f$ is a homomorphism of multiplicative monoids,
\item $\gamma_{B}f(\boxplus^{d}(a_{1},\cdots,a_{2^{d}}))
=f(\gamma_{A})\boxplus^{d}(f(a_{1}),\cdots,f(a_{2^{d}}))$, and
\item $\gamma_{B}$ and $f(\gamma_{A})$
generates the same ideal in $B$.
\end{enumerate}
\item Two $d$-convexoid structures $\boxplus^{d}_{1}$
and $\boxplus^{d}_{2}$ on a $d$-convexoid ring $A$
are \textit{equivalent}, if
the identity map $(A,\boxplus^{d}_{1}) \to (A,\boxplus^{d}_{2})$
is a weak isomorphism.
\end{enumerate}
\end{Def}
It is easy to see that weak homomorphisms
are closed under compositions.

\begin{Prop}
Let $A$ be a $d$-convexoid ring.
\begin{enumerate}
\item $A$ is equivalent to a normalized
$d$-convexoid ring if and only if $\gamma_{A}$ is invertible.
\item In particular if $d=1$,
then $A$ is equivalent to a ring if and only if $\gamma_{A}$
is invertible.
\end{enumerate}
\end{Prop}
\begin{proof}
\begin{enumerate}
\setcounter{enumi}{1}
\item It suffices to show that a normalized
convexoid ring is a ring.
 Associativity of $\boxplus$ is given by
\begin{multline*}
(a\boxplus b) \boxplus c=(a \boxplus b) \boxplus (1 \boxplus 0)c
=(a \boxplus b) \boxplus (c \boxplus 0) \\
=(a \boxplus 0) \boxplus (b \boxplus c)=a \boxplus (b \boxplus c).
\end{multline*}
Also, $0$ is the unit with respect to $\boxplus$:
\[
a=(1 \boxplus 0)a=a \boxplus 0.
\]
\end{enumerate}
\end{proof}

\begin{Cor}
The left adjoint of the underlying functor
$U:\cat{Ring} \to \cat{CxdRing}$ is given by
$R \mapsto R/\equiv$, where
$\equiv$ is the congruence generated by $1 \boxplus 0=1$.
\end{Cor}

\begin{Cor}
\label{cor:ring:loc:cxd}
Let $A$ be a convexoid ring.
Then, $A[\gamma^{-1}_{A}]$ is equivalent to a ring.
\end{Cor}
This implies that, rings are in a sense,
`localizations' of convexoid rings.

\begin{Cor}
\label{cor:initial:cxd:ring}
The initial object of $\cat{CxdRing}$
can be realized as the smallest subset $R_{0}$ 
of the polynomial ring $\ZZ[\gamma]$ satisfying
\begin{enumerate}[(a)]
\item $0,1 \in R_{0}$, and
\item $f,g \in R_{0} \Rightarrow fg, \gamma(f+g), -f \in R_{0}$.
\end{enumerate}
The natural functor $F:\cat{CxdRing} \to \cat{Ring}$
gives a surjective convexoid ring homomorphism
$R_{0} \to \ZZ$ between the initial objects defined
by $\gamma \mapsto 1$.
\end{Cor}
\begin{Lem}
\label{lem:surj:r0:dzz1/2}
\begin{enumerate}
\item $m\gamma^{n} \in R_{0}$ for any non-negative integers $m,n$
such that $|m|\leq 2^{n}$.
\item In particular, the homomorphism
$R_{0} \to \mathbf{D}\ZZ[1/2]$ ($\gamma \mapsto 1/2$)
is surjective.
\end{enumerate}
\end{Lem}
This is an easy calculation, and the proof is left to the reader.

\begin{Prop}
Let $\cat{Mnd$_{0}$}$ be
the category of (multiplicative) monoids
with absorbing elements.
For $M \in \cat{Mnd$_{0}$}$,
we can define the trivial $d$-convexoid structure on $M$
by $\boxplus^{d}(a_{1},\cdots,a_{2^{d}})\equiv 0$ for any 
$a_{1},\cdots,a_{2^{d}} \in M$.
Hence, we have a fully faithful functor
$\cat{Mnd$_{0}$} \to \cat{CxdRing$^{d}$}$.
\end{Prop}
This proposition shows that,
any monoid can be regarded as a convexoid ring.
However, the $\boxplus$-structure has no information
in this case.

Next, we will compare $d$-convexoid rings
and $e$-convexoid rings, where $e$ is a divisor of  $d$.
Let $(A,\boxplus^{e})$ be an $e$-convexoid ring.
Then we can canonically define a $d$-convexoid structure
by induction on $d$:
\begin{multline*}
\boxplus^{d}(a_{1},\cdots,a_{2^{d}})
=\boxplus^{e}(\boxplus^{d-e}(a_{1},\cdots,a_{2^{d-e}}),\cdots,
\boxplus^{d-e}(a_{2^{d}-2^{d-e}+1},\cdots,a_{2^{d}})).
\end{multline*}
\begin{Prop}
Suppose $d=er$ for some positive integers $e,r$.
Let $(A,\boxplus^{d})$ be a $d$-convexoid ring,
and suppose the fundamental constant
$\gamma_{A}$ is a $r$-th power of an invertible element $\mu$.
Then, we can define $el$-convexoid structure on $A$ for $1 \leq l \leq r$
by an descending induction on $l$:
\[
\boxplus^{el}(a_{1},\cdots,a_{2^{el}})
=\mu^{-1}\boxplus^{e(l+1)}(a_{1},\cdots,a_{2^{el}},0,\cdots,0).
\]
Moreover, the $d$-convexoid structure
 $\tilde{\boxplus}^{d}$ induced from $\boxplus^{e}$
 coincides with $\boxplus^{d}$.
\end{Prop}
Note that
\[
\boxplus^{el}(a_{1},\cdots,a_{el})
=\mu^{l-r}\boxplus^{d}(a_{1},\cdots,a_{2^{el}},0,\cdots,0).
\]
\begin{proof}
It is a straightforward calculation that
$\boxplus^{el}$ gives a $el$-convexoid structure on $A$.
We will see that $\tilde{\boxplus}^{d}$
coincides with $\boxplus^{d}$ by induction:
\begin{multline*}
\tilde{\boxplus}^{d}(a_{1},\cdots,a_{2^{d}}) \\
=\mu^{1-r}\boxplus^{d}(
\mu^{-1}\boxplus^{d}(a_{1},\cdots,a_{2^{e(r-1)}},0,\cdots,0),\cdots, \\
\mu^{-1}\boxplus^{d}(a_{2^{d}-2^{e(r-1)}+1},\cdots,a_{2^{d}},0,\cdots,0 ),
0,\cdots,0) \\
=\mu^{-r}\boxplus^{d}(\boxplus^{d}(a_{1},\cdots,a_{2^{d}}),0,\cdots,0)
=\boxplus^{d}(a_{1},\cdots,a_{2^{d}}).
\end{multline*}
\end{proof}

\begin{Def}
\label{def:diff:wt:equiv}
Suppose $e$ is a divisor of a positive integer $d$.
Let $A,B$ be an $e$-convexoid ring
and a $d$-convexoid ring, respectively.
A map $f:A \to B$ is a \textit{weak homomorphism}
if $f$ decomposes into a sequence of maps
\[
A \stackrel{\tilde{f}}{\to}
 B ^{\prime} \stackrel{j}{\to} B
\]
such that
\begin{enumerate}
\item $B^{\prime}$ is a $d$-convexoid ring,
\item $j$ is a weak isomorphism in the sense of Definition 
\ref{def:cxd:weak:hom},
\item the $d$-convexoid structure of $B^{\prime}$
is induced by an $e$-convexoid structure $\boxplus^{e}_{B}$, and
\item $\tilde{f}$ is a weak homomorphism
(in the sense of Definition \ref{def:cxd:weak:hom})
with respect to this $e$-convexoid structure.
\end{enumerate}
A $d$-convexoid structure $\boxplus^{d}$
and an $e$-convexoid structure $\boxplus^{e}$
on $A$ are \textit{equivalent}, if
the identity maps $(A, \boxplus^{e}) \to (A,\boxplus^{m})$
and $(A,\boxplus^{d}) \to (A, \boxplus^{m})$ are
a weak isomorphisms for some $m$-convexoid structure $\boxplus^{m}$,
with $m$ a common multiple of $d$ and $e$.
\end{Def}
We can verify that a composition of
two weak homomorphism becomes again a weak homomorphism.

With the aid of Corollary \ref{cor:ring:loc:cxd},
we obtain:
\begin{Cor}
Let $(A,\boxplus^{d})$ be a $d$-convexoid ring.
Then, $A[\gamma_{A}^{-1}]$ is equivalent to a ring.
\end{Cor}

Next, we will investigate the spectrum
for commutative convexoid rings.
Here, we will restrict our attention to the Zariski topology.

In the sequel, we assume that
any convexoid ring is commutative,
and its fundamental constant is
a non-zero divisor.
 
\begin{Prop}
\label{prop:invar:spec:equiv}
Let $(R, \boxplus^{d})$ be a commutative $d$-convexoid ring,
and $\tilde{\boxplus}^{d}$ another $d$-convexoid structure on $R$.
\begin{enumerate}
\item Suppose $\tilde{\boxplus}^{d}$ is defined by
 $\tilde{\boxplus}^{d}=L_{u} \circ \boxplus^{d}$
 for some $u \in R$, where $L_{u}$ is the left multiplication of $u$.
Then, there is an immersion $\Spec R \to \Spec \tilde{R}$.
\item If $\boxplus^{d}$ and $\tilde{\boxplus}^{d}$ are equivalent
in the sense of Definition \ref{def:cxd:weak:hom},
then ideals of $(R,\boxplus^{d})$ are exactly
those of $(R,\tilde{\boxplus}^{d})$.
In particular, $\Spec (R,\boxplus^{d})$ and 
$\Spec (R,\tilde{\boxplus}^{d})$
are canonically homeomorphic.
\end{enumerate}
In particular for a commutative ring $R$ we have immersions
\[
\Spec^{\cat{Ring}}R \to
\Spec^{\cat{Cxd}}(R,\boxplus) \to 
\Spec^{\cat{CMnd$_{0}$}}R,
\]
where $a \boxplus b=u(a+b)$ is the convexoid structure on $R$
defined by a constant $u \in R$,
and $\Spec^{\cat{Cxd}}(R,\boxplus)$
(resp. $\Spec^{\cat{CMnd$_{0}$}}R$)
is the spectrum obtained by regarding $R$
as a convexoid ring (resp. commutative monoid with 
an absorbing element).
\end{Prop}
\begin{proof}
\begin{enumerate}
\item We have a natural map $\Omega(\tilde{R}) \to \Omega(R)$,
sending $\mathfrak{a}$ to the ideal generated by $\mathfrak{a}$.
This is a surjective lattice homomorphism,
hence induces an immersion $\Spec R \to \Spec \tilde{R}$
of the corresponding morphism of coherent spaces.
\item Let $\gamma,\tilde{\gamma}$ be the fundamental constants
of $(R,\boxplus^{d}),(R,\tilde{\boxplus}^{d})$, respectively.
Since $\boxplus$ and $\tilde{\boxplus}$ are equivalent,
$R[\gamma^{-1}]$ and $R[\tilde{\gamma}^{-1}]$ are equal,
and $\gamma=u\tilde{\gamma}$ for some $u \in R$.
Suppose $\mathfrak{a}$ is an ideal of $(R,\boxplus^{d})$,
and $a_{1},\cdots ,a_{2^{d}}\in \mathfrak{a}$. Then,
\[
\boxplus^{d}(a_{1},\cdots,a_{2^{d}})
=\gamma\sum_{i}a_{i}
=u\tilde{\gamma}\sum_{i}a_{i}
=u\tilde{\boxplus}^{d}(a_{1},\cdots,a_{2^{d}}).
\]
This shows that $\mathfrak{a}$ is also an ideal of $(R,\tilde{\boxplus})$.
\end{enumerate}
\end{proof}

\begin{Prop}
\label{prop:inv:wt:spec}
Let $e$ be a divisor of a positive integer $d=er$,
 $(A,\boxplus^{e})$ an $e$-convexoid ring, and
and $\boxplus^{d}$ the $d$-convexoid structure
induced by $\boxplus^{e}$.
Then, the natural map 
\[
\Omega(A, \boxplus^{e}) \to \Omega(A,\boxplus^{d})
\]
is an isomorphism.
Its inverse is given by $\mathfrak{a} \mapsto \langle \mathfrak{a} \rangle$,
where $\langle \mathfrak{a} \rangle$
is the ideal generated by $\mathfrak{a}$.
\end{Prop}
This shows that the underlying topological space
of the spectrum of a multi-convexoid ring is invariant under
weak isomorphisms (in the sense of Definition \ref{def:diff:wt:equiv}).

\begin{proof}
Let $\mathfrak{a} \in \Omega(A,\boxplus^{d})$
be an finitely generated ideal.
It suffices to show that
\[
\sqrt{\mathfrak{a}}=\{a \in A \mid a^{n} \in \mathfrak{a} \
(\exists n)\}
\]
is an ideal in $(A,\boxplus^{e})$,
hence equal to $\langle \mathfrak{a} \rangle$.
For any $a_{1},\cdots,a_{2^{e}} \in \sqrt{\mathfrak{a}}$,
set
\[
b=\boxplus^{e}(a_{1},\cdots,a_{2^{e}})^{r}
=\boxplus^{d}(a^{p})_{p} \in \mathfrak{a},
\]
where $p$ runs through all maps
$\{1,\cdots,r\} \to 2^{e}$
and $a^{p}=\prod_{i=1}^{2^{e}}a_{i}^{\# p^{-1}(i)}$.
This shows that $b^{N} \in \mathfrak{a}$
for sufficiently large $N$.
\end{proof}

\section{Ostrowski's Theorem}

A commutative convexoid ring is \textit{integral},
if $0$ is a prime ideal.
As in the case of rings, we can define
the fractional field $Q(R)$ of an integral convexoid ring $R$.
A commutative integral convexoid ring $R$
is a \textit{valuation convexoid ring},
if for any non-zero element $x$ of the fractional field $K=Q(R)$,
either $x$ or $x^{-1}$ is in $R$.
As in the case of rings,
the set $I(R)$ of $R$-submodules of $K$ becomes
totally ordered by the inclusion relations.
We have a group homomorphism
\[
|\cdot |:K^{\times} \to I(R) \setminus 0 \quad (a \mapsto aR),
\]
and this satisfies the following properties:
\begin{enumerate}
\item $\ker |\cdot |=R^{\times}$, and
\item $|a\boxplus b| \leq \max \{|a|,|b|\}$.
\end{enumerate}

The following theorem is a variant of
the classical Ostrowski's theorem:
\begin{Thm}
\label{thm:ostrowski}
Let $R$ be a non-trivial valuation convexoid ring 
with $Q(R)$ weak-isomorphic to $\QQ$ and $1 \boxplus 1$ invertible.
Then, $R$ is either
\begin{enumerate}
\item the local ring $\ZZ_{(p)}$ at the finite place $p \neq 2$
with $\boxplus=+$, or
\item the unit disk
$\mathbf{D}\QQ=\{x \in \QQ \mid |x|_{\infty} \leq 1\}$,
with $a \boxplus b=\pm(a+b)/2$.
\end{enumerate}
\end{Thm}
\begin{proof}
Let $|\cdot|$ be the valuation corresponding to $R$.
Since $R$ is a subring of $\QQ$,
the value group $I(R)$ is automatically archimedean.
Hence, we may regard $I(R)$ as a multiplicative submonoid
of positive real numbers.
Also, note that $|\cdot|$ is determined by the value on $\NNN$.
\begin{itemize}
\item \textit{Case $|\NNN| \leq 1$}:
We will show that $R$ is equal to $\ZZ_{(p)}$
for some prime $p$.
It suffices to show that there is a unique prime $p$
such that $|p|<1$.

First, we will show the uniqueness:
suppose there exist two primes $p,q$ such that $|p|,|q|<1$.
Then $|p|^{e},|q|^{e}<|1/2|$ for sufficiently large integer $e$.
Also, there are two integers $m,n$ such that 
\[
\gamma^{-1}(mp^{e}\boxplus nq^{e})=mp^{e}+nq^{e}=1
\]
since $p^{e}$ and $q^{e}$ are coprime.
Since $2\gamma=1\boxplus 1$ is invertible, we have
\begin{multline*}
|1|=|2||1/2\gamma ||mp^{e}\boxplus nq^{e}|
\leq |2|\max \{ |m||p|^{e},|n||q|^{e}\} \\
\leq |2| \max \{ |p|^{e},|q|^{e}\}
<|2||1/2|=|1|,
\end{multline*}
a contradiction. Hence such a prime $p$ is unique.

Since $R$ is non-trivial, there exists an integer $n$
with $|n|<1$. Hence, there is a prime divisor $p$
of $n$ such that $|p|<1$.
\item \textit{Case $|\NNN| \not\leq 1$}:
Let $\boxplus^{m}:R^{2^{m}} \to R$ be 
the $m$-convexoid structure induced from $\boxplus$.
For three integers $a,b,n \in \NNN$ bigger than $1$,
$b^{n}$ can be expanded in the form
\[
b^{n}=\Sum_{0 \leq i <2^{m}}c_{i}a^{i},
\]
where $m$ is an integer satisfying
$|2|^{m-1} \leq n \log_{a}b <|2|^{m}$,
and $c_{i}=\{0,1,\cdots,a-1\}$
and is zero for $i \geq j$ for some $j \leq n\log_{a}b+1$.
Note that $|2|>1$ from the assumption $|\NNN| \not\leq 1$.
Also, there is an integer $l$ such that $a \leq 2^{l}$.
Then 
\[
b^{n}=2^{m}(2\gamma)^{-m}\boxplus^{m}
(c_{1}a,c_{2}a^{2},\cdots,c_{j-1}a^{j-1},0,\cdots0),
\]
and $|c_{i}| \leq |2|^{l}$.
These yield
\begin{multline*}
|b^{n}| \leq |2|^{m}\max_{i}|c_{i}a^{i}|
\leq |2|^{m+l}\max\{|a|^{j},1\} \\
\leq (n\log_{a}b)\cdot |2|^{l+1}\max \{|a|^{n\log_{a}b},1\}.
\end{multline*}
Taking the $n$-th root of both sides, we obtain
\[
|b| \leq (n\log_{a}b\cdot |2|^{l+1})^{1/n}
\max \{|a|^{\log_{a}b},1\}.
\]
Taking the limit $n \to \infty$, we have
$|b| \leq \max \{|a|^{\log_{a}b},1\}$.
Now, take $b$ as $|b| >1$.
Then $|a|$ must also be bigger than $1$,
hence $\log |b|/\log b\leq \log |a|/\log a$.
By symmetry, this inequality is in fact equal.
Hence, the valuation $|\cdot|$ is equivalent
to the absolute norm,
and $|2\gamma|=1$ shows that $\gamma=\pm 1/2$.
\end{itemize}
\end{proof}

\begin{Thm}
\label{thm:ostrowski:general}
The above Ostrowski's theorem is valid
for any algebraic field:
let $K$ be an algebraic field,
and $R$ be a non-trivial valuation convexoid ring
with $1 \boxplus 1$ invertible.
Then, $R$ is either
\begin{enumerate}
\item the local ring $\mathcal{O}_{K,\mathfrak{p}}$,
where the characteristic of the residue field
$\kappa(\mathfrak{p})$ is not $2$, or
\item $\mathbf{D}_{\sigma}K=\{x \in K \mid |\sigma(x)| \leq 1\}$,
where $\sigma:K \to \CC$ is an immersion of fields.
\end{enumerate}
\end{Thm}
\begin{proof}
Here, we will only check the key points
which are different from the proof of 
original Ostrowski's theorem. (cf. \cite{BS}, pp. 278-280)
Let $|\cdot|$ be the valuation associated to $R$.
\begin{enumerate}
\item The valuation $\nu$ associated to $R$
is non-trivial on $\QQ$.
Indeed, suppose the valuation is trivial on $\QQ$, and
let $b_{1},\cdots,b_{n}$ be a basis of $K$
over $\QQ$.
Then it turns out that $\nu(x) \leq 2^{n}\max_{i}\nu(b_{i})$
for any $x \in K$, which is a contradiction
since a non-trivial valuation is never bounded.
\item Let $R$ be a valuation convexoid ring
of $\CC$, which satisfies $R \cap \RR=\mathbf{D}\RR$.
Then $R=\mathbf{D}\CC$.
Indeed, suppose $\nu(z) >1$ for some $z \in \CC$
such that $|z|=1$. Then for any $n$,
\[
\nu(z^{n}) \leq 2\max\{ \Re (z), |\Im (z)|\cdot \nu(\sqrt{-1})\}
\leq 2\max \{1,\nu(\sqrt{-1})\}
\]
which is a contradiction since $\nu(z^{n}) \to \infty$ as $n \to \infty$.
For general $0 \neq z \in \CC$, we have
\[
\nu(z)=\nu(|z|)\nu(z/|z|)=|z|.
\]
\end{enumerate}
The remaining part of the proof is completely
identical to the original one, hence we will omit it.
\end{proof}

\section{Convexoid schemes}

The main goal of this paper is to obtain
the compactification $\overline{\Spec \ZZ}=\Spec \ZZ \cup \{\infty\}$
in the form of Zariski-Riemann space.

We can consider (weak) `convexoid ring schemes'
in the sense of \cite{Takagi1};
the definitions are analogous to those of schemes
and weak $\scr{C}$-schemes.

However, this is not sufficient for our purpose,
since we need to admit weak isomorphisms
for restriction maps and transition maps.

\begin{Def}
A \textit{convexoid scheme}
is a pair $(X,\scr{O}_{X})$ such that
$X$ is a coherent space and 
$\scr{O}_{X}$ is a $\cat{CMnd$_{0}$}$-valued
sheaves, such that
\begin{enumerate}
\item $X$ is locally isomorphic 
(as a monoid-valued space) to the spectrum
of a multi-convexoid ring:
namely, there is a finite open covering $X=\cup_{i}U_{i}$
and isomorphisms $\phi_{i}:\Spec^{\cat{Cxd}} R_{i} \to U_{i}$ 
of monoid-valued spaces, where $R_{i}$
is a commutative multi-convexoid ring.
\item If $\phi_{i}^{-1}(V)$ is affine for some
open subset $V$ of $U_{i} \cap U_{j}$, then $\phi_{j}^{-1}(V)$
is also affine, and the transition map
\[
\sigma_{ji}:\phi_{i}^{-1}(V) \to \phi_{j}^{-1}(V)
\]
is a weak isomorphism.
\item $\sigma_{kj}\sigma_{ji}=\sigma_{ki}$
for any $i,j,k$.
\end{enumerate}
A morphism $f:X \to Y$ of convexoid schemes
is a morphism of monoid-valued spaces
such that the induced homomorphism
\[
\scr{O}_{Y,f(x)} \to \scr{O}_{X,x}
\]
is a weak homomorphism of multi-convexoid rings 
and is \textit{local} for any $x \in X$,
namely the inverse image of the maximal ideal of $\scr{O}_{X,x}$
coincides with that of $\scr{O}_{Y,f(x)}$.
\end{Def}
Here, we defined the appropriate notion
of convexoid schemes for our purpose.
Note that for the set of sections $\scr{O}_{X}(U)$
for an open set $U$ of a convexoid scheme $X$
may not be a multi-convexoid ring, but only a
multiplicative monoid.

Before we construct the compactification
$\overline{\Spec \ZZ}$,
we need to define the base convexoid scheme $S_{0}$.
Let $R_{0}$ be the initial object of $\cat{CxdRing}$,
and set
\[
U_{1}=\Spec R_{0}[\gamma^{-1}],\
U_{2}=\Spec R_{0}[(2\gamma)^{-1}].
\]
These are open subsets of $\Spec R_{0}$.
Note that $R_{0}[\gamma^{-1}]$ has a ring structure,
isomorphic to $\ZZ[\gamma^{\pm 1}]$.

There is an automorphism on
\[
U_{1} \cap U_{2} =\Spec^{\cat{Cxd}} R_{0}[(2\gamma^{2})^{-1}]
\simeq \Spec^{\cat{Cxd}} \ZZ[1/2][\gamma^{\pm 1}]
\]
induced by the automorphism
\[
\phi:\ZZ[1/2][\gamma^{\pm 1}] \to \ZZ[1/2][\gamma^{\pm 1}]
\quad (\gamma \mapsto \gamma/2).
\]
We have a diagram of open immersions
\[
U_{1} \leftarrow U_{1} \cap U_{2} \stackrel{\phi^{*}}{\to}
U_{1} \cap U_{2} \to U_{2},
\]
which gives a `twist' patching $S$ of $U_{1}$ and $U_{2}$,
namely $S$ is defined by the pushout diagram
\[
\xymatrix{
U_{1} \cap U_{2} \ar[r] \ar[d]_{\phi^{*}} & U_{1} \ar[d] \\
U_{2} \ar[r] & S_{0}.
}
\]
We have a closed immersion
$f:\Spec \ZZ \to U_{1}$ defined by 
\[
R_{0}[\gamma^{-1}] \simeq \ZZ[\gamma^{\pm 1}] \to \ZZ \quad 
(\gamma \mapsto 1).
\]

If we stick to convexoid schemes,
we obtain a `fake closure' of $\Spec \ZZ$:
we have a homomorphism
\[
R_{1}=R_{0}[(2\gamma)^{-1}] \to \mathbf{D}\ZZ[1/2]
\quad (\gamma \mapsto 1/2)
\]
which is surjective by Lemma \ref{lem:surj:r0:dzz1/2}.
This induces a morphism
$g:\Spec \mathbf{D}\ZZ[1/2] \to U_{2}$,
which is a closed immersion since the complement of the image is
\[
\Spec R_{1}[(\gamma\boxplus (-1/2))^{-1}]
=\Spec R_{1}[(\gamma-1/2)^{-1}].
\]
\begin{Lem}
\label{lem:spec:disk:zz1/2}
The underlying space of the spectrum
$\Spec \mathbf{D}\ZZ[1/2]$ is set-theoretically
isomorphic to $\Spec \ZZ[1/2] \cup \{\infty\}$,
where $\infty$ is the pullback of the maximal ideal
$\{ a \in \QQ \mid |a|<1\}$ of $\mathbf{D}\QQ$.
As a coherent space, $\infty$ is the unique closed point.
\end{Lem}
\begin{proof}
The inclusion $\Spec \ZZ[1/2] \cup \{\infty\}
 \subset \Spec \mathbf{D}\ZZ[1/2]$ is obvious.
 Let $\mathfrak{p}$ be a prime ideal of $\Spec \mathbf{D}\ZZ[1/2]$.
 If $\mathfrak{p}$ does not contain $1/2$,
 then $\mathfrak{p}$ is a pullback of a prime ideal of $\ZZ[1/2]$.
Suppose $\mathfrak{p}$ contain $1/2$.
Then $\mathfrak{p}$ contains any $a \in \mathbf{D}\ZZ[1/2]$
such that $|a| \leq 1/2$.
For any element $a \in \mathbf{D}\ZZ[1/2]$
 with its absolute value less than $1$,
we have $|a^{n}|<1/2$ for sufficiently large $n$.
Since $\mathfrak{p}$ is prime, $a$ must be in $\mathfrak{p}$.
It is obvious that this $\mathfrak{p}$ is the unique maximal ideal,
since its complement is the unit group $\{\pm 1\}$.
\end{proof}

Using the above lemma,
we define $Y_{0}$ by the pushout:
\[
\xymatrix{
\Spec \ZZ[1/2] \ar[r] \ar[d] & \Spec \ZZ \ar[d] \\
\Spec \mathbf{D}\ZZ[1/2] \ar[r] & Y_{0}.
}
\]
We have a commutative diagram
\[
\xymatrix{
\Spec \ZZ[1/2] \ar[r]^{f} \ar@{=}[d] & 
U_{1} \cap U_{2} \ar[d]^{\phi^{*}} \\
\Spec \ZZ[1/2] \ar[r]_{g} & U_{1} \cap U_{2},
}
\]
which shows that $f$ and $g$ patch up to give
a closed immersion $Y_{0} \to S_{0}$.
This can be regarded as the closure of $\Spec \ZZ$ in $S_{0}$
with its reduced induced subscheme structure.

\begin{Rmk}
\label{rmk:global:sec}
The monoid of sections $\scr{O}_{Y_{0}}(U)$
for an open set $U$ of $Y_{0}$ becomes a convexoid ring,
if and only if $U$ does not contain both two points $2$ and $\infty$:
$2$ and $\infty$ are antipodes to each other.
In particular, 
$\Gamma(Y_{0},\scr{O}_{Y_{0}})=\FF_{1^{2}}=\{0,\pm 1\}$
is only a multiplicative monoid.
This happens since the local fundamental constant
$1/2$ and $1$ (they are defined on $\Spec \mathbf{D}\ZZ[1/2]$
and $\Spec \ZZ$, respectively) cannot be extended globally.

However, since the infinity place $\infty$ of $\Spec \mathbf{D}\ZZ[1/2]$
is the maximal ideal, we cannot have $\mathbf{D}\QQ$
on the stalk $\scr{O}_{Y_{0},\infty}$;
we have $\mathbf{D}\ZZ[1/2]$ instead.
This is not what we want.
\end{Rmk}

\section{Graded convexoid rings and $\Proj$}

In the previous section,
we gave an ad hoc definition of the `fake closure' $Y_{0}$
of $\Spec \ZZ$.
At first sight, this seems to be a very artificial object.
But some reader may have noticed that it resembles
to the construction of projective line $\PP^{1}$
in algebraic geometry.
Indeed, $Y_{0}$ can be realized as $\Proj R_{0}$,
under a suitable definition.

\begin{Def}
\begin{enumerate}
\item A convexoid ring $A$ is \textit{\rom{(}$\NNN$-\rom{)}graded}, if:
\begin{enumerate}
\item $B=A[\gamma_{A}^{-1}]$ is a ($\ZZ$-)graded ring:
say $B=\oplus_{d \in \ZZ}B_{d}$,
where $B_{d}$ is the degree $d$ part.
\item $A$ is in the positive part:
$A \subset \oplus_{d \geq 0}B_{d}$.
\item For any $a \in A$,
$a_{d}$ is also in $A$ for any $d$,
where $a=\sum_{d}a_{d}$ is the homogeneous decomposition
of $a$ in $B$.
\end{enumerate}
We will denote $A \cap B_{d}$ by $A_{d}$,
and the set of homogeneous elements of $A$ by $A^{h}$.
Note that $A_{d} \neq A \cap B_{d}$ in general.
\item An ideal $\mathfrak{a}$ of a graded convexoid ring 
is \textit{homogeneous}, if $a \in \mathfrak{a}$
implies $a_{d} \in \mathfrak{a}$ for any $d$,
where $a=\sum a_{d}$ is the homogeneous decomposition of $a$.
\item For any graded convexoid ring $A$,
set $A_{+}=A \cap (\oplus_{d>0}B_{d})$.
\end{enumerate}
\end{Def}
We will assume the following convention for
any graded convexoid ring $A$:
\begin{eqnarray}
\label{graded:cxd:ring:1st}
\textit{$A_{+}$ is finitely generated as a homogeneous ideal of $A$.}
\end{eqnarray}
For a homogeneous element $a$,
its degree will be denoted by $|a|$.

\begin{Def}
Let $A$ be a commutative graded convexoid ring.
\begin{enumerate}
\item $\Proj A$ is the set of all homogeneous
prime ideals of $A$ which does not contain $A_{+}$.
The open basis of $\Proj A$ is given by the form
\[
D_{+}(f)=\{\mathfrak{p} \in \Proj A \mid 
f \notin \mathfrak{p}\},
\]
where $f$ is a homogeneous element of $A$.
Then $\Proj A$ becomes a coherent space,
and has a finite open covering $\Proj A=\cup_{f \in A^{h}} D_{+}(f)$,
by the assumption (\ref{graded:cxd:ring:1st}).
\item For $f \in A_{d}$, a $d$-convexoid ring $A_{(f)}$ is defined by
\[
A_{(f)}=\left\{ a/f^{n} \in A[f^{-1}] \mid 
\rom{$a$ is homogeneous of degree $dn$}\right\}.
\]
The $d$-convexoid structure on $A_{(f)}$ is given by
\[
\boxplus^{d}\left(\Frac{a_{1}}{f^{n_{1}}}, \cdots, 
\Frac{a_{2^{d}}}{f^{n_{2^{d}}}}\right)
=\Frac{1}{f^{N+|\gamma_{A}|}}
\boxplus^{d}_{A}(a_{1}f^{m_{1}},\cdots,a_{2^{d}}f^{m_{2^{d}}}),
\]
where $N=\sum_{i=1}^{2^{d}} n_{i}$
and $m_{i}=N-n_{i}$.
\item We have a homeomorphism
$D_{+}(f) \to  \Spec A_{(f)}$
given by $\mathfrak{p} \mapsto 
\{a/f^{n} \in A_{(f)} \mid a \in \mathfrak{p}\}$.
Its inverse is given by 
$\mathfrak{q} \mapsto \tilde{\mathfrak{q}}$,
where $\tilde{\mathfrak{q}}$
is the homogeneous ideal of $A$ generated by $a \in A^{h}$
such that $a^{|f|}/f^{|a|} \in \mathfrak{q}$.

\item We can define a monoid-valued sheaf $\scr{O}$
on $\Proj A$, so that $(D_{+}(f) ,\scr{O}|_{D_{+}(f)})$
is isomorphic to $\Spec A_{(f)}$ as a monoid-valued space.
This is well defined, since we have a natural isomorphism
$\varphi:A_{(f)}[f^{|g|}/g^{|f|}] \simeq A_{(g)}[g^{|f|}/f^{|g|}]$
of multiplicative monoids
for any two homogeneous elements $f,g \in A^{h}$.
Also, $(\Proj A,\scr{O})$ becomes a convexoid scheme,
since $\varphi$ is a weak isomorphism of multi-convexoid rings.
\end{enumerate}
\end{Def}
\begin{Rmk}
Note that $A_{(f)}$ does not have
a convexoid structure in general when $|f| >1$,
since the $\boxplus$-operation shifts the degree.
\end{Rmk}

We will apply $\Proj$ to the initial object $R_{0}$.
Note that $R_{0}$ is graded, by setting $\deg \gamma=1$.
\begin{Lem}
The degree $1$ part $(R_{0})_{1}$
of $R_{0}$ consists of $\pm \gamma,\pm 2\gamma$.
Also, $(R_{0})_{+}$ is generated by $(R_{0})_{1}$.
\end{Lem}
Again, this is an easy exercise, and the proof is left to the reader.

By the above lemma, $\Proj R_{0}$ is covered by
two affines 
\[
\begin{split}
D_{+}(\gamma) \simeq \Spec A_{(\gamma)} \simeq \Spec \ZZ, \\
D_{+}(2\gamma) \simeq \Spec A_{(2\gamma)} \simeq \Spec \mathbf{D}\ZZ[1/2].
\end{split}
\]
It is obvious to see that the patching of $D_{+}(\gamma)$
and $D_{+}(2\gamma)$ coincides with that of $Y_{0}$ introduced
in the previous section,
which shows that $Y_{0}$ is isomorphic to $\Proj R_{0}$.

\begin{Rmk}
As we have mentioned in Remark \ref{rmk:global:sec},
the global section $\Gamma(\Proj R_{0}, \scr{O})$ of $\Proj R_{0}$
is $\FF_{1^{2}}$, which is only a monoid.
We might want to formulate a morphism 
$\pi:\Proj A \to \Spec \FF_{1^{2}}$ in some sense
and say that $\pi$ is proper,
but there lies a technical difficulty:
since $\infty$ is the unique closed point in $\Spec \mathbf{D}\ZZ[1/2]$,
there are two homomorphisms 
$\Spec \ZZ_{(p)} \to \Spec\mathbf{D}\ZZ[1/2]$ for odd prime $p$,
sending the closed point to either $\infty$ or $(p) \in \ZZ[1/2]$.
This shows that we cannot say that $\Proj A$ is proper.
\end{Rmk}

\section{Weak convexoid schemes}

The appropriate compactification
$\overline{\Spec \ZZ}$ 
is realized only as a more general object,
the construction of which
is given by an analogue of that of $\scr{A}$-schemes \cite{Takagi2}.


For a coherent space $X$,
we denote by $\Omega(X)$ the distributive lattice
of quasi-compact open subsets of $X$.
There is a natural $\cat{DLat}$-valued sheaf $\tau_{X}$ on $X$
defined by $U \mapsto \Omega(U)$ for each quasi-compact open $U$.

\begin{Def}
\begin{enumerate}
\item A \textit{weak convexoid scheme}
is a quadruple $(X,\scr{O}_{X},\mathcal{U},\beta_{X})$
where $X$ is a coherent space,
$\scr{O}_{X}$ is a sheaf of commutative multiplicative monoids
on $X$, and $\mathcal{U}=\{(U_{i},\boxplus_{i}^{d_{i}})\}_{i}$
is a set of pairs of a quasi-compact open subset $U_{i}$ of $X$
and a $d_{i}$-convexoid ring structure $\boxplus_{i}^{d_{i}}$
 on $\scr{O}_{X}(U_{i})$,
such that
\begin{enumerate}
\item $\mathcal{U}$ is a covering of $X$,
namely $\cup_{U \in \mathcal{U}}U=X$, and
\item $\mathcal{U}$ is a lower set,
namely if $V \subset U$ and $U \in \mathcal{U}$,
then $V \in \mathcal{U}$ and $\Gamma(U) \to \Gamma(V)$
is a weak homomorphism.
\end{enumerate}
We can define a $\cat{DLat}$-valued sheaf $\Omega\scr{O}_{X}$
as follows:
recall that the correspondence $R \mapsto \Omega(R)$
(see the end of \S 0 for the definition)
gives a functor $\Omega:\cat{CxdRing} \to \cat{DLat}$,
where $\cat{DLat}$ is the category of distributive lattices.
Note that $\Omega(R)$ does not change
when we replace the $\boxplus$ by 
another equivalent multi-convexoid structure by
Proposition \ref{prop:invar:spec:equiv}
and Proposition \ref{prop:inv:wt:spec}.

For an open subset $V$ of $X$,
let $\mathcal{U}_{V}$ be the subset of $\mathcal{U}$
consisting of all quasi-compact open subsets contained in $V$.
Then $\Omega\scr{O}_{X}(V)$ is defined as the equalizer of
\[
\prod_{U \in \mathcal{U}}\Omega\Gamma (U,\scr{O}_{X})
\rightrightarrows \prod_{U_{1},U_{2} \in \mathcal{U}}
\Omega\Gamma(U_{1} \cap U_{2}).
\]
$\beta_{X}$ is a morphism 
$\Omega\scr{O}_{X} \to \tau_{X}$ 
of $\cat{DLat}$-valued sheaves on $X$, which satisfies the following:
for any inclusion $V \subset U$ of open subsets of $X$,
the restriction map $\scr{O}_{X}(U) \to \scr{O}_{X}(V)$
factors through $T^{-1}\scr{O}_{X}(U)$,
where $T$ is the multiplicative system of $\scr{O}_{X}(U)$
defined by
\[
T=\{f \mid \beta_{X}(U)(f) \geq V\}.
\]
Here, $f$ is identified with the principal ideal 
$(f) \in \Omega\scr{O}_{X}(U)$ generated by $f$.
We refer to $\beta_{X}$ as the \textit{support morphism of $X$}.
\item For a weak convexoid scheme $X$
and a point $x \in X$,
the stalk $\scr{O}_{X,x}$ need not have
a canonical choice of a $\boxplus$-structure.
However, we can define the notion of a 
finitely generated radical ideal of $\scr{O}_{X,x}$:
it is independent of the choice of the $\boxplus$-structure.
Also, $\scr{O}_{X,x}[\gamma_{x}^{-1}]$
has a natural structure of a commutative ring,
where $\gamma_{x}$ is the fundamental constant
of any $\scr{O}_{X}(U)$, $x \in U$.
(This constant depends on the choice of $U$,
but the localization $\scr{O}_{X,x}[\gamma_{x}^{-1}]$ is
independent. Hence, we will call
$\gamma_{x}$ \textit{the fundamental constant of $\scr{O}_{X,x}$}.)
Then $\scr{O}_{X,x}$ becomes \textit{local},
in the sense that the complement of the set of units 
forms the maximal ideal.

\item A morphism $f:X \to Y$ of weak convexoid schemes
is a morphism of monoid-valued spaces such that
for any $x \in X$,
\begin{enumerate}
\item
$f_{x}:\scr{O}_{Y,f(x)} \to \scr{O}_{X,x}$
induces a ring homomorphism
$\scr{O}_{Y,f(x)}[\gamma_{f(x)}^{-1}] \to \scr{O}_{X,x}[\gamma_{x}^{-1}]$,
and
\item $f_{x}$ is \textit{local}:
$f_{x}^{-1}(\mathfrak{m}_{x})=\mathfrak{m}_{f(x)}$
where $\mathfrak{m}_{x}$ (resp. $\mathfrak{m}_{f(x)}$)
is the unique maximal ideal of $\scr{O}_{X,x}$
(resp. $\scr{O}_{Y,f(x)}$).
\end{enumerate}
\end{enumerate}
\end{Def}

We will first recall what a Zariski-Riemann space should be.

\begin{Def}
\begin{enumerate}
\item 
A morphism $f:X \to S$ of convexoid schemes
is \textit{proper}, if it satisfies the following condition:
for any commutative square
\[
\xymatrix{
\Spec K \ar[r] \ar[d] & X \ar[d]^{f} \\
\Spec R \ar[r] \ar@{.>}[ur] & S
}
\]
where $R$ is a valuation convexoid ring and $K$
its fraction field, there exist a unique morphism
$\Spec R \to X$ making the whole diagram commutative.

\item Let $X$ be a (weak) convexoid scheme over 
a base (weak) convexoid scheme $S$.
The \textit{Zariski-Riemann space of $X$ over $S$}
 is a $S$-morphism $X \to \ZR(X,S)$
where $\ZR(X,S)$ is a proper (weak) convexoid scheme over $S$
and is universal:
namely, any $S$-morphism $f:X \to Y$ with
$Y \to S$ proper factors uniquely through $\ZR(X,S)$:
\[
\xymatrix{
X \ar[r]^{f} \ar[d] & Y \\
\ZR(X,S) \ar@{.>}[ur]
}
\]
\end{enumerate}
\end{Def}
If the Zariski-Riemann space exists,
then it is unique up to isomorphism.
However, this may not be constructed
within the category of (weak) convexoid schemes.

Let $\mathcal{O}_{K}$ be the integer ring
of an algebraic field $K$.
Note that the Zariski-Riemann space $\ZR(\Spec K,\Spec\ZZ)$
is isomorphic to $\Spec \mathcal{O}_{K}$.
Since $\Spec \ZZ$ is a closed subscheme of $U_{1}$,
we see that $\ZR(\Spec K,U_{1})$ coincides
with $\Spec \mathcal{O}_{K}$.

We will prove the following:
\begin{Thm}
The Zariski-Riemann space 
\[
X=\ZR(\Spec K,S_{0})=\ZR(\Spec \mathcal{O}_{K},S_{0})
=\ZR(\Spec \mathcal{O}_{K},\Proj R_{0})
\]
exists as a weak convexoid scheme.
Its underlying space
is set-theoretically isomorphic to 
$\Spec \ZZ \cup \{\infty_{\sigma}\}_{\sigma}$,
where $\infty_{\sigma}$ is the absolute valuation
corresponding to an immersion $\sigma:K \to \CC$
of fields.
The stalk $\scr{O}_{X,\infty}$ is isomorphic
to $\mathbf{D}_{\sigma}K=\{x \in K \mid |\sigma(x)|_{\infty} \leq 1\}$.
\end{Thm}
This is what we wanted to construct.
\begin{Rmk}
Let $\scr{A}$ be the algebraic type of commutative rings.
In the category of (profinite) $\scr{A}$-schemes,
the existence of the Zariski-Riemann space is assured
\cite{Takagi2}.
However, we do not have a general
theory of Zariski-Riemann spaces for convexoid schemes.
Therefore, we will content ourselves
by constructing the Zariski-Riemann space $X$
explicitly for this specific case.
\end{Rmk}

\begin{proof}
First, we will construct $X$.
We only have to construct its restriction $X_{2}$
to the fiber on $U_{2}$,
since Zariski-Riemann spaces are local with respect to the base.
Let $|X_{2}|$ be the set of all valuation convexoid rings
of $K$ such that $2\gamma=1 \boxplus 1$ is invertible.
We endow a topology on $|X_{2}|$ which is generated
by the open basis of the form
\[
U(\mathcal{S})=\{R \in |X_{2}| \mid \mathcal{S} \subset R \},
\]
where $\mathcal{S}$ is any finite subset of $K$.

The theorem of Ostrowski \ref{thm:ostrowski:general}
tells that any valuation convexoid ring in $|X_{2}|$
is either 
\begin{enumerate}[(a)]
\item the trivial one, 
\item the (non-complete) discrete valuation ring
$\mathcal{O}_{K,(\mathfrak{p})}$ such that 
the characteristic of the residue field $\kappa(\mathfrak{p})$
is not $2$,
\item the disk $\mathbf{D}_{\sigma}K$ associated to 
an immersion $\sigma:K \to \CC$ of fields.
\end{enumerate}
The above topology makes $|X_{2}|$ into a coherent space:
a non-empty open subset $U$ of $|X_{2}|$
is a subset whose complement is a finite set
not containing the trivial valuation ring.

The structure sheaf $\scr{O}_{X}|_{X_{2}}$
is defined by 
\[
U \mapsto \{a \in K \mid a \in R \quad (\forall R \in U)\}.
\]
The support morphism 
$\beta_{X}:\Omega\scr{O}_{X}|_{X_{2}} \to \tau_{X}|_{X_{2}}$
is defined by
\[
(f_{1},\cdots ,f_{n}) \mapsto \{ R \in U \mid
f_{i} \in \mathfrak{m}_{R} \ (\forall i) \}^{c}.
\]
It is straightforward to see that $\beta_{X}$
is well defined and that $X_{2}=(|X_{2}|,\scr{O}_{X}|_{X_{2}},
\beta_{X})$ becomes a weak convexoid scheme.
Let us denote by $\infty_{\sigma}$ the point of $X_{2}$
corresponding to $\mathbf{D}_{\sigma}K$.
Then we see that $X_{2} \setminus \{\infty_{\sigma}\}_{\sigma}$
is isomorphic to $\Spec \mathcal{O}_{K}[1/2]$.
Therefore, we obtain $X$ by the pushout
\[
\xymatrix{
\Spec \mathcal{O}_{K}[1/2] \ar[r] \ar[d] & X_{2} \ar[d] \\
\Spec \mathcal{O}_{K} \ar[r] & X.
}
\]
We have a convexoid ring homomorphism
\[
R_{0}[(2\gamma)^{-1}] 
\to \Gamma(X_{2}, \scr{O}_{X})=
\cap_{\sigma:K \to \CC}\mathbf{D}_{\sigma}\mathcal{O}_{K}[1/2]
\]
by $\gamma \mapsto 1/2$.
This induces a morphism $X_{2} \to U_{2}$,
and patches up with $\Spec \mathcal{O}_{K} \to U_{1}$
to give the morphism $\nu:X \to S_{0}$.
We see that $\nu$ is proper,
since $\nu|_{\Spec \mathcal{O}_{K}}$ is a closed immersion,
and $\nu|_{X_{2}}$ is obviously proper
from the construction.

Finally, we will see that $X$ has the universal
property. It suffices to show that $X_{2} \to U_{2}$
satisfies the property.
Let $f:\Spec\mathcal{O}_{K}[1/2] \to Y$ be a $U_{2}$-morphism,
where $Y$ is a weak convexoid scheme,
proper over $U_{2}$.
For each valuation convexoid ring $R \in |X_{2}|$,
we have the following commutative diagram
\[
\xymatrix{
\Spec K \ar[r] \ar[d] & Y \ar[d] \\
\Spec R \ar[r] \ar@{.>}[ru]& U_{2},
}
\]
and the properness of $Y$ tells that there
is a unique arrow $\Spec R \to Y$ making
the whole diagram commutative.
This gives a unique set-theoretic map $\tilde{f}:|X_{2}| \to |Y|$.
This becomes continuous, since
it is continuous on $X_{2} \setminus \{\infty_{\sigma}\}_{\sigma} \to Y$.
It remains to construct the morphism
between the structure sheaves.
We only have to consider 
$\scr{O}_{Y}(U) \to \scr{O}_{X_{2}}(\tilde{f}^{-1}U)$,
when $\tilde{f}^{-1}U$ contains some 
infinite places $\infty_{\sigma}$.
Since $\infty_{\sigma} \in \tilde{f}^{-1}U$ implies
that the map $\scr{O}_{Y}(U) \to K$ factors through 
$\mathbf{D}_{\sigma}K$,
this weak homomorphism also factors
through $\scr{O}_{X_{2}}(\tilde{f}^{-1}U)$.
Therefore, we have constructed the morphism
$\tilde{f}:X_{2} \to Y$ of monoid-valued spaces,
and it is straightforward to check that
this is indeed a morphism of weak convexoid schemes.
The uniqueness of $\tilde{f}$ is obvious from the construction.
\end{proof}

\section{Appendix: Embedding of $\Proj R_{0}$}

As we have seen, the initial object
$R_{0}$ in the category of convexoid rings
has a natural grading structure, and the convexoid scheme
$\Proj R_{0}$ is the `fake closure' of $\Spec \ZZ$.

Once we have a projective scheme,
algebraic geometers would ask what the projective embedding
associated to a very ample line bundle might be.
We will seek for an analogy of the projective embedding for $\Proj R_{0}$.
This can be realized, and the result can be summarized as follows:
\begin{Thm}
Let $R_{0}$ be the initial object in the
category of convexoid rings and $d$ a positive integer.
\begin{enumerate}
\item Each line bundle $\scr{O}(d)$
of $\Proj R_{0}$ gives a morphism 
$\Proj R_{0} \to \PP^{2^{d}-1}_{\FF_{1^{2}}}$
of monoid-valued spaces.
\item In particular, we have an immersion
$\Proj R_{0} \to \PP$ into a proprojective space 
$\PP$ over $\FF_{1^{2}}$.
\end{enumerate}
\end{Thm}

Before we proceed,
we will review what the projective space $\PP_{\FF_{1}}^{n}$
(or, $\PP_{\FF_{1^{2}}}^{n}$; the construction is essentially
the same) is.
As we have mentioned in the introduction,
$\FF_{1}$-algebras are regarded as a monoid:
for example, a polynomial ring 
$\FF_{1}[x_{1},\cdots,x_{n}]$ over $\FF_{1}$
is a free commutative monoid $\NNN^{n} \cup \{0\}$ 
with an absorbing element $0$, generated by 
$x_{1},\cdots,x_{n}$.
In this sense,
we know that schemes over $\FF_{1}$ 
(namely, `monoid schemes')
can be constructed, and there is an adjunction
\[
\Spec:\cat{CMnd$_{0}$} \rightleftarrows \cat{Sch/$\FF_{1}$}^{\op}:
\Gamma.
\]
(cf. \cite{Takagi1}).
In the sequel, we only consider
coherent schemes and quasi-compact morphisms.
We have a left adjoint of the underlying functor
$U:\cat{CRing} \to \cat{CMnd$_{0}$}$,
which we denote by $\ZZ[\cdot]/0$:
for $M \in \cat{CMnd$_{0}$}$,
$\ZZ[M]/0$ is the monoid ring $\ZZ[M]$
divided by the ideal generated by the absorbing element $0_{M}$ of $M$.
The functor $\ZZ[\cdot]/0$ patches up to give
a functor 
\[
\ZZ \times_{\FF_{1}} (-):\cat{Sch/$\FF_{1}$} \to \cat{Sch}.
\]
The unit morphism $M \to \ZZ[M]/0$
induces a morphism $\Spec \ZZ[M]/0 \to \Spec M$,
and this extends to give a natural morphism
\[
\pi_{X}:\ZZ \times_{\FF_{1}} X \to X
\]
of monoid-valued spaces.

For example, a fan (in the sense of toric geometry
\cite{Oda}) $\Delta$
together with an absorbing element 
gives a scheme $\Spec \Delta$ over $\FF_{1}$,
and $X=\ZZ \times_{\FF_{1}}\Spec \Delta$ is just the toric scheme 
over $\ZZ$ associated to the fan $\Delta$.
The $\FF_{1}$-scheme $\Spec \Delta$
has its underlying space as a subset of $X$
consisting of the generic points
of the images of $\mathbb{T}$-invariant sections
$\Spec \ZZ \to X$,
where $\mathbb{T}$ is the maximal torus of $X$.
In particular, for each fiber $F$ of $X \to \Spec \ZZ$,
the points of $\Spec \Delta$ correspond
to $\mathbb{T}$-invariant points of $F$.
The natural morphism $\pi_{X}:X \to \Spec \Delta$
sends each point $x$ of $X$ to the generic point
of the closure of the $\mathbb{T}$-orbit of $\overline{\{x\}}$.

In particular, the projective space $\PP^{n}_{\FF_{1}}$
corresponds to the fan $\Delta$ representing the
projective space $\PP^{n}$,
and its points correspond to prime ideals generated by monomials
over homogeneous coordinates;
therefore, the configuration of points
can be described as a $n$-simplex;
each $l$-dimensional face of the $n$-simplex corresponds
to a $l$-dimensional point of $\PP^{n}_{\FF_{1}}$
(Figure \ref{figure:pp2:ff1}).

\begin{figure}
\begin{center}
\unitlength 0.1in
\begin{picture}( 21.1000, 13.3300)( -0.3000,-14.9000)
%
{\color[named]{Black}{%
\special{pn 8}%
\special{pa 800 1400}%
\special{pa 2000 1400}%
\special{fp}%
\special{pa 2000 1400}%
\special{pa 1400 400}%
\special{fp}%
\special{pa 1400 400}%
\special{pa 800 1400}%
\special{fp}%
}}%
%
{\color[named]{White}{%
\special{pn 0}%
\special{sh 1.000}%
\special{ia 1400 406 66 66  0.0000000  6.2831853}%
}}%
{\color[named]{Black}{%
\special{pn 8}%
\special{pn 8}%
\special{ar 1400 406 66 66  0.0000000  6.2831853}%
}}%
%
{\color[named]{White}{%
\special{pn 0}%
\special{sh 1.000}%
\special{ia 806 1400 66 66  0.0000000  6.2831853}%
}}%
{\color[named]{Black}{%
\special{pn 8}%
\special{pn 8}%
\special{ar 806 1400 66 66  0.0000000  6.2831853}%
}}%
%
{\color[named]{White}{%
\special{pn 0}%
\special{sh 1.000}%
\special{ia 2000 1400 60 60  0.0000000  6.2831853}%
}}%
{\color[named]{Black}{%
\special{pn 8}%
\special{pn 8}%
\special{ar 2000 1400 60 60  0.0000000  6.2831853}%
}}%
\put(14.0500,-9.9500){\makebox(0,0){$(0)$}}%
\put(17.2500,-8.7500){\makebox(0,0)[lb]{$(x_{0})$}}%
\put(10.8000,-8.7000){\makebox(0,0)[rb]{$(x_{2})$}}%
\put(14.0000,-15.2000){\makebox(0,0){$(x_{1})$}}%
\put(13.9500,-2.3000){\makebox(0,0){$(x_{0},x_{2})$}}%
\put(20.8000,-14.8000){\makebox(0,0)[lt]{$(x_{0},x_{1})$}}%
\put(8.4500,-14.9000){\makebox(0,0)[rt]{$(x_{1},x_{2})$}}%
\end{picture}%
\end{center}
\caption[]{Configuration of the points of $\PP^{2}_{\FF_{1}}$.}
\label{figure:pp2:ff1}
\end{figure}

If we replace $\FF_{1}$ by $\FF_{1^{2}}$,
then the underlying space does not change,
but only the structure sheaf becomes the
sheaf of $\FF_{1^{2}}$-algebras, namely
each section admits its minus.

Now, we go back to $\Proj R_{0}$.
We will imitate the construction of projective
morphisms in algebraic geometry,
with an exception that we forget the additive structures.

The set $L_{d}=(R_{0})_{d}$ of 
homogeneous elements of $R_{0} \subset \ZZ[\gamma]$
of degree $d$ consists of
\[
0,\pm \gamma^{d}, \pm 2\gamma^{d}, \pm 3\gamma^{d}
\cdots,\pm 2^{d}\gamma^{d}.
\]
Unlike the case of rings, this set does not
have an additive structure, but only the 
$\FF_{1^{2}}=(R_{0})_{0}$-action;
namely, $L_{d}$ is an $\FF_{1^{2}}$-module.
However, we can still regard it as a
linear system, and consider the line bundle $\scr{O}(d)$
and even a rational map, associated to $L_{d}$ as follows.

The line bundle $\scr{O}(d)$
is a $\scr{O}_{\Proj R_{0}}$-submodule
of the locally constant sheaf $\QQ^{*}$, generated by
\[
m\gamma^{d}/n\gamma^{d}=m/n \quad (1\leq m,n \leq 2^{d}).
\]
Indeed, this canonically becomes a line bundle, and
$L_{d}$ can be regarded as the set of global
sections of $\scr{O}(d)$.
In other words, $\scr{O}(d)$ is globally generated.

The linear system $L_{d}$ is free as an $\FF_{1^{2}}$-module,
hence we will fix a basis $\gamma^{d},
2\gamma^{d},\cdots, 2^{d}\gamma^{d}$
to construct the morphism associated to $L_{d}$ in the sequel.

Let $\FF_{1^{2}}[x_{1},\cdots,x_{2^{d}}]$
be a polynomial ring (in fact, a monoid)
with coefficients in $\FF_{1^{2}}$,
with the canonical grading.
For each $1 \leq n \leq 2^{d}$,
we have a morphism of monoids
\[
\FF_{1^{2}}[x_{1},\cdots,x_{2^{d}}]
\to (R_{0})_{(n\gamma^{d})}
\quad
(x_{l} \mapsto l\gamma^{d}/n\gamma^{d}=l/n),
\]
which extends to $\FF_{1^{2}}[x_{1}/x_{n},\cdots,x_{2^{d}}/x_{n}]
\to (R_{0})_{(n\gamma^{d})}$.
These patch up to give a morphism
$f_{d}:\Proj R_{0} \to \PP^{2^{d}-1}_{\FF_{1^{2}}}$
of monoid-valued spaces.
For a finite place $p \in \Spec \ZZ \subset \Proj R_{0}$,
$f_{d}$ sends $p$ to the
point corresponding to the prime
$(x_{p},x_{2p},\cdots,x_{[2^{d}/p]p})$,
where $\{x_{i}\}_{i}$ are homogeneous coordinates of $\PP^{2^{d}-1}$.
The infinity place $\infty$ goes to the point
corresponding to $(x_{1},\cdots,x_{2^{d}-1})$,
which is one of the closed points of $\PP^{2^{d}-1}$.

Note that $f_{d}$ never becomes an immersion,
since a finite place $p$ goes to the generic point
of $\PP^{2^{d}-1}$ when $p$ is larger than $2^{d}$.
This is just one translation of the fact that the multiplicative monoid
$\ZZ \setminus \{0\}$ is not finitely generated.

However, we can consider the infinite product
\[
\PP=\prod_{d}\PP^{2^{d}-1}=
\PP^{2^{1}-1}_{\FF_{1^{2}}} \times_{\FF_{1^{2}}} 
\PP^{2^{2}-1}_{\FF_{1^{2}}} \times_{\FF_{1^{2}}}
\cdots
\]
and a morphism $f:\Proj R_{0} \to \PP$
in the category of \textit{weak} schemes over $\FF_{1}$
(cf. \cite{Takagi1}).
Then, $f$ becomes an immersion.

\textbf{Acknowledgements}:
\index{ach@xach}
I was partially supported by the Grant-in-Aid for Young
Scientists (B) \# 23740017.

\textsc{S. Takagi: Department of Mathematics, Faculty of Science,
Kyoto University, Kyoto, 606-8502, Japan}

\textit{E-mail address}: \texttt{takagi@math.kyoto-u.ac.jp}

\end{document}